\documentclass[11pt]{article}

\usepackage{amssymb, enumerate}
\usepackage{amsmath}
\usepackage{amsthm}
\usepackage{layout}

\newcounter{itemcounter}
\numberwithin{itemcounter}{section}

\newtheorem{thm}[itemcounter]{Theorem}
\newtheorem{lem}[itemcounter]{Lemma}
\newtheorem{defi}[itemcounter]{Definition}
\newtheorem{prop}[itemcounter]{Proposition}
\newtheorem{cor}[itemcounter]{Corollary}

\newtheorem{con}[itemcounter]{Conjecture}

\newtheorem{rem}[itemcounter]{Remark}

\newtheorem*{thm*}{Theorem}
\newtheorem*{con*}{Conjecture}
\newtheorem*{cor*}{Corollary}
\newtheorem*{ack*}{Acknowledgements}

\newcommand{\Irr}{\mathop{\rm Irr}\nolimits}

\newcommand{\Tr}{\mathop{\rm Tr}\nolimits}

\newcommand{\Hom}{\mathop{\rm Hom}\nolimits}

\newcommand{\Out}{\mathop{\rm Out}\nolimits}

\newcommand{\LL}{\mathop{\rm LL}\nolimits}

\newcommand{\rad}{\mathop{\rm rad}\nolimits}

\newcommand{\OF}{\mathop{\rm f_\mathcal{O}}\nolimits}
\newcommand{\SOF}{\mathop{\rm sf_\mathcal{O}}\nolimits}
\newcommand{\mf}{\mathop{\rm mf}\nolimits}
\newcommand{\mfO}{\mathop{\rm mf_\mathcal{O}}\nolimits}

\newcommand{\nth}{\mathop{\rm th}\nolimits}

\newcommand{\GL}{\operatorname{GL}}

\newcommand{\Gal}{\operatorname{Gal}}
\newcommand{\length}{\operatorname{length}}

\newcommand{\cX} {\mathcal{X}}
\newcommand{\cE} {\mathcal{E}}

\newcommand{\cO} {\mathcal{O}}

\newcommand{\NN} {\mathbb{N}}
\newcommand{\ZZ}{\mathbb Z}
\newcommand{\QQ}{\mathbb Q}
\newcommand{\CC}{\mathbb C}
\newcommand{\FF}{\mathbb F}

\def\bigtimes{\mathop{\mathchoice 
 {\hbox{\sf\Large\lower 0.1\baselineskip\hbox{X}}}%
 {\hbox{\sf\large\lower 0.1\baselineskip\hbox{X}}}%
 {\hbox{\sf\normalsize\lower 0.1\baselineskip\hbox{X}}}%
 {\hbox{\sf\tiny\lower 0.1\baselineskip\hbox{X}}}%
}}

\begin{document}

\title{Donovan's conjecture, blocks with abelian defect groups and discrete valuation rings \footnote{This research was supported by the EPSRC (grant nos. EP/M015548/1 and EP/M02525X/1).} }

\author{Charles W. Eaton\footnote{School of Mathematics, University of Manchester, Manchester, M13 9PL, United Kingdom. Email: charles.eaton@manchester.ac.uk}, Florian Eisele\footnote{School of Mathematics and Statistics, University of Glasgow, University Place, Glasgow, G12 8QQ, United Kingdom. Email: florian.eisele@glasgow.ac.uk} and Michael Livesey\footnote{Friedrich-Schiller-Universit\"{a}t Jena, Fakult\"{a}t f\"{u}r Mathematik und Informatik, Institut f\"{u}r Mathematik, 07737 Jena, Germany. Email: michael.livesey@uni-jena.de}}

\date{21st May, 2019}
\maketitle


\begin{abstract}
We give a reduction to quasisimple groups for Donovan's conjecture for blocks with abelian defect groups defined with respect to a suitable discrete valuation ring $\cO$. Consequences are that Donovan's conjecture holds for $\cO$-blocks with abelian defect groups for the prime two, and that, using recent work of Farrell and Kessar, for arbitrary primes Donovan's conjecture for $\cO$-blocks with abelian defect groups reduces to bounding the Cartan invariants of blocks of quasisimple groups in terms of the defect.

A result of independent interest is that in general (i.e. for arbitrary defect groups) Donovan's conjecture for $\cO$-blocks is a consequence of conjectures predicting bounds on the $\cO$-Frobenius number and on the Cartan invariants, as was proved by Kessar for blocks defined over an algebraically closed field.
\end{abstract}


\section{Introduction}


Let $p$ be a prime and let $k=\bar \FF_p$. Let $(K,\cO,k)$ be a $p$-modular system, so $\cO$ is a complete discrete valuation ring with residue field $k$. The results here hold in this general setting, but we have in mind for $\cO$ the ring of Witt vectors over $k$ as this will be used to state Donovan's conjecture in a uniform way. Donovan's conjecture, originally stated over an algebraically closed field, is as follows:

\begin{con}[Donovan]
\label{Donovan:conj}
Let $P$ be a finite $p$-group. Then amongst all finite groups $G$ and blocks $B$ of $\cO G$ with defect groups isomorphic to $P$ there are only finitely many Morita equivalence classes.
\end{con}

For blocks defined over an algebraically closed field there has been some success in proving the conjecture for certain $p$-groups, for example it is known for abelian $2$-groups by~\cite{el18b} and for abelian defect groups in arbitrary characteristic it reduces to bounding the Cartan invariants of blocks of quasisimple groups by~\cite{fk18} and~\cite{el18b}. In addition blocks with defect groups isomorphic to dihedral or semidihedral $2$-groups were classified in~\cite{er90}. However, we ultimately want to understand blocks defined over $\cO$. Two of the main obstacles to working over $\mathcal{O}$ rather than $k$ are as follows. The first is that the crucial reduction step in~\cite{ku95}, allowing us to reduce to studying groups generated by the defect groups, was only known over a field. The second is that the reduction in~\cite{ke04} of Donovan's conjecture into two distinct conjectures was also only known over a field. The first problem was overcome by the second author in~\cite{eis18}, and we resolve the second here, allowing us to reduce Donovan's conjecture for $\cO$-blocks with abelian defect groups to bounding, for quasisimple groups, the Cartan invariants and strong Frobenius number as defined in~\cite{el18a}. The results of~\cite{fk18} show that the strong Frobenius numbers of quasisimple groups are bounded in terms of the defect group, so Donovan's conjecture for abelian defect groups in fact reduces to bounding Cartan invariants of blocks of quasisimple groups. Such bounds are known to hold for $2$-blocks with abelian defect groups.

Our main result is as follows:

\begin{thm}
\label{reduce:theorem}
Let $d$ be a non-negative integer. If there are functions $s,c:\mathbb{N}\to\mathbb{N}$ such that for all $\cO$-blocks $B$ of quasisimple groups with abelian defect groups of order $p^{d'}$ dividing $p^d$, $\SOF(B) \leq s(d')$ and all Cartan invariants are at most $c(d')$, then Donovan's conjecture holds for $\cO$-blocks with abelian defect groups of order $p^d$.
\end{thm}

A straightforward consequence is that:

\begin{cor}
Donovan's conjecture (over $\mathcal{O}$) holds for blocks with abelian defect groups if and only if it holds for blocks of quasisimple groups with abelian defect groups.
\end{cor}

By work of Farrell and Kessar in~\cite{fk18}, we get a much more powerful consequence:

\begin{cor}
Let $d$ be a non-negative integer. If there is a function $c:\mathbb{N}\to\mathbb{N}$ such that for all $\cO$-blocks $B$ of quasisimple groups with abelian defect groups of order $p^{d'}$ dividing $p^d$ the Cartan invariants are at most $c(d')$, then Donovan's conjecture holds for $\cO$-blocks with abelian defect groups of order $p^d$.
\end{cor}

Hence we have shown that for abelian $p$-groups Conjecture \ref{Donovan:conj} is equivalent to (the restriction to quasisimple groups of) the following apparently much weaker conjecture, which arose from a question of Brauer:

\begin{con}[Weak Donovan]
\label{Weak_Donovan:conj}
Let $P$ be a finite $p$-group. Then there is $c(P) \in \NN$ such that if $G$ is a finite group and $B$ is a block of $kG$ with defect groups isomorphic to $P$, then the entries of the Cartan matrix of $B$ are at most $c(P)$.
\end{con}

\begin{rem}
It actually suffices to bound the Cartan invariants of quasisimple groups $G$ with $O_p(G)=1$, as we will see in Section \ref{Cartan}.
\end{rem}

In~\cite{ekks14} it was shown that the Cartan invariants are bounded in terms of the defect group for $2$-blocks with abelian defect groups, so we get:

\begin{thm}
Let $P$ be an abelian $2$-group. Then Donovan's conjecture holds for $P$.
\end{thm}

The paper is structured as follows. In Section~\ref{strong_Frobenius} we recall the definition of the strong $\cO$-Frobenius number and state some of the main results. In Section \ref{Kessar_analogue} we show that Donovan's conjecture for $\cO$-blocks is equivalent to two separate conjectures as in~\cite{ke04}. Section \ref{reductions} contains the reduction to quasisimple groups. In Section \ref{Cartan} we briefly discuss the problem of bounding Cartan invariants.

\emph{Remark on choice of $\mathcal{O}$ in Donovan's conjecture}: Note that since $\cO/J(\cO)$ is algebraically closed we ensure that $K$ contains all $p'$-roots of unity. In general $\cO$ would have to contain a primitive $|P|^{\nth}$ root of unity in order for $K$ to be a splitting field for a block with defect group $P$, but this condition is not always necessary to demonstrate Donovan's conjecture. We therefore have two canonical choices for $\cO$ in the statement of Donovan's conjecture: the ring of Witt vectors for $\bar \FF_p$ and the same with a primitive $|P|^{\nth}$ root of unity attached. The results of this paper hold over either choice (see Remark \ref{O-remark} for the latter case), but in light of the results of~\cite{fk18} the former seems the best setting for Donovan's conjecture.


\section{Strong $\cO$-Frobenius and $\cO$-Morita-Frobenius numbers}
\label{strong_Frobenius}


The strong $\cO$-Frobenius number was introduced in~\cite{el18a}, but we recall the definition and some of its main properties here. We also define the $\cO$-Morita-Frobenius number. These numbers may be defined for any choice of $\cO$ as in the introduction, although this requires some care when it comes to defining the character idempotents.

Let $A$ be an $\cO$-algebra finitely generated as an $\cO$-module. Write $kA$ for $A \otimes_\cO k$ and $KA$ for $A \otimes_\cO K$. Let $G$ be a finite group and $B$ a block of $\cO G$. Denote by $e_B\in\cO G$ the block idempotent corresponding to $B$ and $e_{kB}\in kG$ the block idempotent corresponding to $kB$. Write $\Irr(G)$ for the set of irreducible characters of $G$ and $\Irr(B)$ for the subset of $\Irr(G)$ of irreducible characters $\chi$ such that $\chi(e_B)\neq0$. For each $\chi\in\Irr(G)$ we denote by $e_\chi\in\overline{K}G$ the character idempotent corresponding to $\chi$, where $\overline{K}$ denotes the algebraic closure of $K$. Note that $\overline{K}B=\bigoplus_{\chi\in\Irr(B)}\overline{K}Ge_\chi$. If $X$ and $Y$ are finitely generated $R$-algebras for $R \in \{ K,\cO ,k,\overline{K}\}$, we write $X\sim_{\operatorname{Mor}} Y$ if the categories of finitely generated $X$ and $Y$-modules are (Morita) equivalent as $R$-linear categories.

Denote by $\pi$ a generator of the maximal ideal of $\cO$. Let $\sigma \in \Gal(K/\QQ_p)$ be such that $\sigma(\pi)=\pi$ and $\sigma$ induces a non-trivial automorphism $\bar{\sigma}$ on $\cO/\pi\cO=k$.

Define $A^{(\sigma)}$ to be the $\cO$-algebra with the same underlying ring structure as $A$ but with a new action of the scalars given by $\lambda.a=\lambda^{\sigma^{-1}}a$, for all $\lambda\in \cO$ and $a\in A$. We may similarly define $(kA)^{(\bar{\sigma})}$. We note that, through the identification of elements, $A$ and $A^{(\sigma)}$ are isomorphic as rings but not necessarily as $\cO$-algebras.

In the case that $\bar{\sigma}$ is the Frobenius automorphism given by $x \mapsto x^q$, where $q$ is a power of $p$, it is sometimes convenient to write $-^{(q)}$ for $\bar{\sigma}$. If $B$ is a block of $\cO G$, for some finite group $G$, then we can also define $B^{(q)}$ to be $B^{(\sigma)}$, where $\sigma$ is some lift of $-^{(q)}$. We define $B^{(q)}$ in an alternative way below. In particular we show that it is independent of the choice of $\sigma$.

For a general $\sigma$ we have $\cO G^{(\sigma)}\cong \cO G$ and we can realise this isomorphism via:
\begin{align*}
\cO G^{(\sigma)}&\to \cO G\\
\sum_{g\in G}\alpha_gg&\mapsto\sum_{g\in G}\sigma(\alpha_g)g.
\end{align*}
If $B$ is a block of $\cO G$, then we identify $B^{(\sigma)}$ with its image under the above isomorphism. We can do analogous identifications with $kG$ and its blocks.

Let $q=p^z$ for some $z\in \mathbb Z$. By an abuse of notation we use $-^{(q)}$ to also denote the field automorphism of the universal cyclotomic extension of $\mathbb{Q}$ defined by $\omega_p\omega_{p'}\mapsto\omega_p\omega_{p'}^{q}$, for all $p^{\nth}$-power roots of unity $\omega_p$ and $p'^{\nth}$ roots of unity $\omega_{p'}$. If $\chi\in\Irr(G)$, then we define $\chi^{(q)}\in\Irr(G)$ to be given by $\chi^{(q)}(g)=\chi(g)^{(q)}$ for all $g\in G$. If $B$ is a block of $\cO G$ with $\chi\in\Irr(B)$, then we define $B^{(q)}$ to be the block of $\cO G$ with $\chi^{(q)}\in\Irr(B^{(q)})$. We have $(kB)^{(q)}=k(B^{(q)})$, in particular $B^{(q)}$ is well-defined. Note that if $\sigma:\cO\to\cO$ is a lift of $-^{(q)}$, then $B^{(q)}= B^{(\sigma)}$.

\begin{defi}
\label{morita_frobenius:def}
Let $A$ be a finitely generated $k$-algebra.
\begin{enumerate}[(i)]
\item The \textbf{Morita Frobenius number} $\mf(A)$ of $A$ is the smallest integer $n$ such that $A\sim_{\operatorname{Mor}}A^{(p^n)}$ as $k$-algebras.
\end{enumerate}
Let $B$ a block of $\cO G$, for some finite group $G$.
\begin{enumerate}
\item[(ii)] The \textbf{$\cO$-Morita Frobenius number} $\mfO(B)$ of $B$ is the smallest integer $n$ such that $B\sim_{\operatorname{Mor}} B^{(p^n)}$ as $\cO$-algebras.

\item[(iii)] The \textbf{$\cO$-Frobenius number} $\OF(B)$ of $B$ is the smallest integer $n$ such that $B\cong B^{(p^n)}$ as $\cO$-algebras.

\item[(iv)] The \textbf{strong $\cO$-Frobenius number} $\SOF(B)$ of $B$ is the smallest integer $n$ such that there exists an $\cO$-algebra isomorphism $B\to B^{(p^n)}$ such that the induced $\overline{K}$-algebra isomorphism $\overline{K}B\to \overline{K}B^{(p^n)}$ sends $\chi$ to $\chi^{(p^n)}$ for all $\chi\in\Irr(B)$.
\end{enumerate}
\end{defi}

Note that the definition of strong $\cO$-Frobenius number given above is not exactly the same as that given in~\cite[Definition 3.8]{el18a} but the two are shown to be equivalent in~\cite[Proposition 3.5]{el18a}.

A consequence of the following is that bounding the strong $\cO$-Frobenius numbers in terms of the size of the defect group is equivalent to bounding the $\cO$-Morita-Frobenius numbers in terms of the size of the defect group.

\begin{prop}
\label{morita_O_proposition}
Let $G$ and $H$ be finite groups, and let $B$ and $C$ be blocks of $\cO G$ and $\cO H$ respectively. Let $D$ be a defect group for $B$.

\begin{enumerate}[(i)]
\item $\mf(kB) \leq \mfO(B) \leq \OF (B) \leq \SOF(B)\leq|D|^2!\OF(B)$.
\item If $B$ and $C$ are Morita equivalent, then $\SOF(B)=\SOF(C)$ and $\mfO(B)=\mfO(C)$.

\end{enumerate}
\end{prop}

\begin{proof}
\begin{enumerate}[(i)]
\item The first three inequalities should be clear from the definitions and the final inequality is in~\cite[Proposition 3.11]{el18a}.
\item The first part is~\cite[Proposition 3.12]{el18a} and the second is immediate from the definition.
\end{enumerate}
\end{proof}

We state an analogue of~\cite[Conjecture 1.3]{ke04}:

\begin{con}
\label{O-Morita-Frobenius:conj}
Let $P$ be a finite $p$-group. Then there is $s(P) \in \NN$ such that if $G$ is a finite group and $B$ is a block of $\cO G$ with defect groups isomorphic to $P$, then $\SOF(B) \leq s(P)$.

Equivalently, there is $t(P) \in \NN$ such that if $G$ is a finite group and $B$ is a block of $\cO G$ with defect groups isomorphic to $P$, then $\mfO(B) \leq t(P)$.
\end{con}


\section{Morita-Frobenius numbers and Donovan's conjecture}
\label{Kessar_analogue}

As in the rest of the paper, the results of this section hold over any complete discrete valuation ring $\mathcal{O}$ with residue field $k=\bar \FF_p$, but we have in mind the ring of Witt vectors of $k$. Denote by $\pi$ a generator of the maximal ideal of $\cO$. Let us fix an element $\sigma$ of $\Gal(K/\QQ_p)$ such that $\sigma(\pi)=\pi$ and $\sigma$ induces a positive power of the Frobenius automorphism on $\cO/\pi\cO$.
If $\cO$ is the ring of Witt vectors over $k$, then $\pi=p$ and any power of the Frobenius automorphism of $k$ can explicitly  be lifted to $\cO$.
 We denote the automorphism of $k$ that $\sigma$ induces by $\bar \sigma$.
The ultimate aim of the section is to prove an analogue over $\cO$ for Kessar's results in~\cite{ke04} which hold over $k$.

Defining ``$-^{\langle \sigma\rangle}$'', resp.``$-^{\langle \bar{\sigma}\rangle}$'', to be the elements fixed under $\sigma$, resp. $\bar{\sigma}$, the field $ k^{\langle\bar{\sigma}\rangle}$ is finite by definition and we claim that $(K^{\langle\sigma\rangle}, \cO^{\langle\sigma\rangle}, k^{\langle\bar{\sigma}\rangle})$ is again a $p$-modular system. It is clear that $K^{\langle\sigma\rangle}$ is complete and that $\cO^{\langle\sigma\rangle}$ is integrally closed in $K^{\langle\sigma\rangle}$. Moreover $\cO^{\langle\sigma\rangle} /\pi \cO^{\langle\sigma\rangle} \subseteq k^{\langle\bar{\sigma}\rangle}$. We just need to check that this inclusion is an equality. To see this, note that every non-zero element of $k^{\langle \bar{\sigma} \rangle}$ is a $(|k^{\langle\bar{\sigma}\rangle}|-1)^{\nth}$ root of unity, and those lift to $\cO$ by Hensel's lemma. That is, the groups of $(|k^{\langle\bar{\sigma}\rangle}|-1)^{\nth}$ roots of unity of $\cO$ and $k$ are in bijection (via reduction mod $\pi$), and since $\bar{\sigma}$ acts trivially on the latter, $\sigma$ must act trivially on the former. Hence $(|k^{\langle\bar{\sigma}\rangle}|-1)^{\nth}$ roots of unity in $\cO$ lie in $\cO^{\langle\sigma\rangle}$ and reduce to the non-zero elements of $k^{\langle\bar{\sigma}\rangle}$, so the claim is shown.

	\begin{defi}[Order]
		We call an $\cO$-algebra $\Lambda$ an \emph{$\cO$-order} if it is free and finitely generated as an $\cO$-module. By an $\cO$-order in a finite-dimensional $K$-algebra $A$ we mean an $\cO$-order contained in $A$ which, in addition, spans $A$ as a vector space over $K$.
	\end{defi}

    \begin{prop}[Lang's theorem over $\cO$]\label{prop lang over o}
        Let $m \in \NN$ and extend $\sigma$ to $K^{m\times m}$ entry-wise:
        $$
            \sigma:\ K^{m\times m} \longrightarrow K^{m\times m}:\ (a_{i,j})_{i,j} \mapsto (\sigma(a_{i,j}))_{i,j}
        $$
        Then the map
        $$
            \GL_m(\cO) \longrightarrow \GL_m(\cO):\ A \mapsto A^{-1} \cdot \sigma(A)
        $$
        is surjective.
    \end{prop}
    \begin{proof}
         Note that the restriction of the epimorphism
        $$\bar{\phantom{A}}:\ \cO^{m\times m} \longrightarrow k^{m\times m}:\ A \mapsto A + \pi \cdot \cO^{m\times m}$$
        to $\GL_m(\cO)$ induces an epimorphism $\GL_m(\cO)\longrightarrow \GL_m(k)$. Therefore, given a matrix $A\in \GL_m(\cO)$, Lang's theorem gives us a matrix $B\in \GL_m(k)$ such that $\bar A = B^{-1}\cdot \bar{\sigma}(B)$, where $\bar{\sigma}(B)$ is the image of $B$ with $\bar{\sigma}$ applied entry-wise. Choose $C_1\in \GL_m(\cO)$ such that $\bar C_1 = B$. Then clearly
        $
            A - C_1^{-1} \cdot \sigma (C_1) \in \pi\cdot \cO^{m\times m}
        $.
        Now let $n\in\mathbb{N}$ and assume there exist $C_i\in\GL_m(\cO)$ for each $1\leq i\leq n$ satisfying
        $
            A - C_i^{-1} \cdot \sigma (C_i) \in \pi^i\cdot \cO^{m\times m}
        $
        for each $1\leq i\leq n$ and
        $
            C_i-C_{i+1} \in \pi^i\cdot \cO^{m\times m}
        $
        for each $1\leq i\leq n-1$. Then, for any $X\in \cO^{m\times m}$:
        $$
            \begin{array}{rl}
            &(C_n-\pi^n\cdot X\cdot C_n)^{-1}\cdot \sigma (C_n-\pi^n \cdot X\cdot C_n)
            \\\\=&  ((1-\pi^n\cdot X)\cdot C_n)^{-1} \cdot (1-\pi^n\cdot \sigma(X))\cdot \sigma(C_n)
             \\\\=& C_n^{-1}\cdot (1-\pi^n\cdot X)^{-1} \cdot (1-\pi^n\cdot \sigma(X))\cdot \sigma(C_n)\\\\
            =& \displaystyle C_n^{-1} \cdot \left(\sum_{j=0}^{\infty} \pi^{n\cdot j}\cdot X^j \right) \cdot (1-\pi^n\cdot \sigma(X)) \cdot \sigma(C_n).
            \end{array}
        $$
        Hence
        $$
            (C_n-\pi^n\cdot X\cdot C_n)^{-1}\cdot \sigma (C_n-\pi^n \cdot X\cdot C_n)
            \equiv C_n^{-1} \cdot (1+\pi^n\cdot (X - \sigma(X))) \cdot \sigma(C_n)\  (\textrm{mod } \pi^{n+1}).
        $$
        If we set
        $
            C_{n+1}:=C_n-\pi^n\cdot X\cdot C_n,
        $
        then we have
        $
            A - C_{n+1}^{-1} \cdot \sigma (C_{n+1}) \in \pi^{n+1}\cdot \cO^{m\times m}
        $
        if and only if $X$ satisfies
        $$
            C_n\cdot A \cdot \sigma(C_n)^{-1} \equiv 1+\pi^n\cdot (X-\sigma(X))
            \ (\textrm{mod }\pi^{n+1}).
        $$
        The same congruence mod $\pi^n$ is satisfied by assumption. Thus we can rewrite this as
        $$
            \pi^{-n}\cdot (C_n\cdot A \cdot \sigma(C_n)^{-1}-1) \equiv X - \sigma(X) \ (\textrm{mod }\pi).
        $$
        We can find such an $X$ once we show that the map
        $$
            k^{m\times m} \longrightarrow k^{m\times m}: (x_{i,j}) \mapsto (x_{i,j}-\bar{\sigma}(x_{i,j})) = (x_{i,j}-x_{i,j}^q)
        $$
        is surjective. Surjectivity of this map is equivalent to the statement that the polynomial equation $x-x^q-z=0$ has a solution in $k$ for any $z \in k$. Since $k$ is algebraically closed, such a solution always exists. Therefore, by induction, there exist $C_i\in\GL_m(\cO)$ for each $i\in\mathbb{N}$ satisfying
        $
            A - C_i^{-1} \cdot \sigma (C_i) \in \pi^i\cdot \cO^{m\times m}
        $
        and
        $
            C_i-C_{i+1} \in \pi^i\cdot \cO^{m\times m}
        $
        for each $i\in\mathbb{N}$. Therefore, since $\cO$ is complete with respect to $(\pi)$, there exists some $C\in\GL_m(\cO)$ (the limit of the $C_i$'s) such that $ A=C^{-1} \cdot \sigma (C)$.
    \end{proof}

    \begin{thm}\label{thm ex o0 form}
        Let $\Lambda$ be an $\cO$-order. Set $K_0=K^{\langle \sigma\rangle}$ and $\cO_0=\cO^{\langle \sigma\rangle}$.
        If there is an isomorphism of $\cO$-algebras
        $$
            \Phi:\ \Lambda \longrightarrow {\Lambda^{(\sigma)}}
        $$
        then there exists an $\cO_0$-algebra $\Lambda_0\subseteq \Lambda$ such that $\Lambda \cong \cO \otimes_{\cO_0} \Lambda_0$.
    \end{thm}
    \begin{proof}
        As a set $\Lambda^{(\sigma)}$ is equal to $\Lambda$, and hence we may view $\Phi$ as a $\sigma$-sesquilinear map from $\Lambda$ into itself. Now fix an isomorphism of $\cO$-lattices $\Delta:\ \Lambda \longrightarrow \cO^n$, where $n=\operatorname{rank}_\cO(\Lambda)$. Let $F:\ \cO^n \longrightarrow \cO^n$ denote the $\sigma$-sesquilinear map given by entry-wise application of $\sigma$.
         Then the map
        $$
            \Delta\circ \Phi \circ \Delta^{-1} \circ F^{-1}: \cO^n\longrightarrow \cO^n
        $$
        is $\cO$-linear (being the composition of a $\sigma$-sesquilinear and a $\sigma^{-1}$-sesquilinear map), and may therefore be viewed as an element of $\GL_n(\cO)$ (as all maps involved in its construction are bijections). Now  Lemma~\ref{prop lang over o} implies that there is an $A\in \GL_n(\cO)$ such that
        $$
            \Delta\circ \Phi \circ \Delta^{-1} \circ F^{-1} = A^{-1}\circ \sigma(A) = A^{-1}\circ F\circ A \circ F^{-1}.
        $$
        The above equation implies that
        \begin{equation}\label{eqn phi f}
        A\circ \Delta\circ\Phi\circ \Delta^{-1}\circ A^{-1}=F.
        \end{equation}
        Let $e_1,\ldots,e_n$ denote the standard basis of $\cO^n$, and set $\lambda_i=\Delta^{-1}(A^{-1}(e_i))$ for $1\leq i \leq n$. Since $F(e_i)=e_i$ for all $i$, formula \eqref{eqn phi f} implies that $\Phi(\lambda_i)=\lambda_i$ for all $i$.

        Next, let us define structure constants $m_{i,j;l}\in \cO$ for $i,j,l\in\{1,\ldots,n\}$ such that
        $$
        \lambda_i\cdot\lambda_j = \sum_{l=1}^n m_{i,j;l} \cdot \lambda_l \quad \textrm{for all $i,j\in\{1,\ldots,n\}$}.
        $$
        The $\sigma$-sesquilinearity of $\Phi$ implies that
        $$
        \Phi(\lambda_i \cdot \lambda_j) = \Phi(\sum_{l=1}^n m_{i,j;l} \cdot \lambda_l) = \sum_{l=1}^n \sigma(m_{i,j;l}) \cdot \lambda_l.
        $$
        The fact that $\Phi$ is multiplicative (by virtue of being an algebra isomorphism between $\Lambda$ and $\Lambda^\sigma$) implies that
        $$
        \Phi(\lambda_i \cdot \lambda_j) = \Phi(\lambda_i)\cdot \Phi(\lambda_j) =
        \lambda_i\cdot\lambda_j = \sum_{l=1}^n m_{i,j;l} \cdot \lambda_l.
        $$
        Since the $\lambda_i$ are linearly independent it follows that $m_{i,j;l}=\sigma(m_{i,j;l})$ for all $i,j,l\in\{1,\ldots,n\}$, i.e. $m_{i,j;l}\in \cO_0$. This implies that the $\cO_0$-lattice spanned by $\lambda_1,\ldots,\lambda_n$ is an $\cO_0$-algebra, which completes the proof.
    \end{proof}

In the following we have in mind the case $K_0=K^{\langle \sigma\rangle}$.

    \begin{prop}\label{prop finitely many cyclotomic split}
        Given a finite extension $K_0/\QQ_p$, and a natural number $n$, there are only finitely many isomorphism classes of semi-simple
        $K_0$-algebras $A$ of dimension $n$.
    \end{prop}
    \begin{proof}
        The Artin-Wedderburn theorem implies that it suffices to prove that there are only finitely many division algebras $A$ of a given dimension $n$ over $K_0$.
        As a $Z(A)$-algebra, a skew-field $A$ is determined by its Hasse invariant (see~\cite[\S 14]{re75}), which can take only finitely many possible values once we fix $\dim_{Z(A)}(A)$.
        Hence we are reduced to showing that there are only finitely many possibilities for the field $Z(A)$, that is, that there are only finitely many field extensions of $K_0$ of degree at most $n$. But it is well known that the number of extensions of $\QQ_p$ of a fixed degree is finite (see~\cite[Th\'{e}or\`{e}me 2]{kr66}, which even gives an explicit formula). Clearly the same holds for extensions of $K_0$, as $K_0$ is of finite degree over $\QQ_p$. This completes the proof.
    \end{proof}

In what follows we denote by $\length_{R}(M)\in \ZZ_{\geq 0}\cup\{\infty\}$, for a commutative local ring $R$ and $R$-module $M$, the length of $M$ as an $R$-module. We will also allow more flexibility for the choice of $K_0$. We will often ask that $K_0/\mathbb Q_p$  be an \emph{extension of discretely valued fields}, which means that it should be a (not necessarily finite) field extension such that the usual discrete (exponential) valuation $\nu_p:\ \mathbb Q_p\longrightarrow \ZZ\subset \QQ$
satisfying $\nu_p(p)=1$ extends to a discrete valuation $K_0\longrightarrow \QQ$, also denoted by $\nu_p$.
It is well known that the valuation on $\mathbb Q_p$ can be extended (even uniquely) to any  algebraic extension of finite degree. But $K/\mathbb Q_p$ is an extension of discretely valued fields as well, after appropriate rescaling of the valuation. To be explicit, we let $\nu_p:\ K \longrightarrow \QQ$ denote the unique discrete valuation on $K$ such that $\nu_p(p)=1$. If we equip $K$ with this valuation, $K/\QQ_p$ becomes an extension of discretely valued fields.

    \begin{prop}
        Let $K_0/\QQ_p$ be an extension of discretely valued fields, let $\cO_0$ be the associated discrete valuation ring, and let $A$ be a split semisimple $K_0$-algebra.
        We have
        $$
            A \cong \bigoplus_{i=1}^n K_0^{d_i\times d_i}
        $$
        for certain $n, d_1,\ldots,d_n\in \NN$. Denote by $\Tr_i:\ A \longrightarrow K_0$ the trace function on the $i^{\nth}$ matrix algebra summand  of $A$.
        Fix elements $u_1,\ldots,u_n\in K_0^\times $ and define
        $$
            T:\ A \times A \longrightarrow K_0:\ (a,b) \mapsto \sum_{i=1}^n u_i\cdot \Tr_i(a\cdot b).
        $$

        If $\Lambda\subset A$ is an $\cO_0$-order such that
        $$
            \Lambda = \Lambda^\sharp := \{ a \in A \ | \ T(a, x)\in \cO_0 \textrm{ for all $x\in \Lambda$} \}
        $$
        and $\Gamma \supseteq \Lambda$ is a maximal $\cO_0$-order in $A$, then
        \begin{equation}\label{eqn length estimate}
            \length_{\cO_0} \Gamma/\Lambda = \frac{1}{2}\cdot \length_{\cO_0}(\cO_0/p\cO_0)\cdot \sum_{i=1}^n d_i^2 \cdot \nu_p(u_i^{-1}).
        \end{equation}
    \end{prop}
    \begin{proof}
        All maximal orders in $A$ are conjugate. Moreover, any conjugate of $\Lambda$ is self-dual with respect to the same bilinear form $T$, that is, $(a\Lambda a^{-1})^\sharp=a\Lambda a^{-1}$ for any $a\in A^\times$. This is because the trace functions $\Tr_i$ used in the definition of $T$ are invariant under conjugation. Hence we may assume without loss of generality that $$\Gamma = \bigoplus_{i=1}^n \cO_0^{d_i\times d_i}.$$
        Using the matrix units as an explicit basis of $\Gamma$ we can compute
        $$\Gamma^{\sharp}=\bigoplus_{i=1}^n u_i^{-1} \cdot \cO_0^{d_i\times d_i}.$$
        Moreover, $T$ induces a non-degenerate pairing (with values in $K_0/\cO_0$, the quotient of the underlying additive group of $K_0$ by the underlying additive group of $\cO_0$) between the $\cO_0$-modules $\Gamma/\Lambda$ and $\Lambda^\sharp/ \Gamma^\sharp=\Lambda/\Gamma^\sharp$. It follows that these two $\cO_0$-modules have the same length, which must consequently be exactly half the length of $\Gamma/\Gamma^\sharp$. The asserted formula for the length of $\Gamma/\Lambda$ now follows immediately.
    \end{proof}

    \begin{defi}[Defect of a symmetric order] \label{defi defect}
        Let $K_0/\QQ_p$ be an extension of discretely valued fields and let $\cO_0$ be the associated discrete valuation ring.
        \begin{enumerate}
        \item  Let $A$ be a split semisimple $K_0$-algebra.          We have
                $$
                    A \cong \bigoplus_{i=1}^n K_0^{d_i\times d_i}
                $$
        for certain $n, d_1,\ldots,d_n\in \NN$.
        If $\Lambda\subseteq A$ is a symmetric $\cO_0$-order, then there is a symmetrising form
        $$
            T:\ A \times A \longrightarrow K_0:\ (a,b) \mapsto \sum_{i=1}^n u_i\cdot \Tr_i(a\cdot b)
        $$
        for certain $u_1,\ldots,u_n\in K_0^\times$ such that $\Lambda=\Lambda^\sharp$ (see \cite[Definition (III.1)]{ple83} for an introduction to symmetrising forms from this point of view). We call
        $$
            d= \max_{1\leq i \leq n} \{ -\nu_p(u_i) \}
        $$
        the \emph{defect of $\Lambda$}.
        \item Assume now that $A$ is an arbitrary semisimple $K_0$-algebra and that, as in the previous point, $\Lambda \subseteq A$ is a symmetric $\cO_0$-order. Then there is an algebraic extension  $E_0/K_0$ of finite degree such that $E_0\otimes_{K_0}A$ is split.
        As the extension is of finite degree, the discrete  valuation of $K_0$ extends uniquely to a discrete valuation on $E_0$.
         If $\mathcal E_0$ denotes the valuation ring of $E_0$, then $\cE_0\otimes_{\cO_0} \Lambda$ is an $\cE_0$-order in the split semisimple $E_0$-algebra $E_0\otimes_{K_0}A$, and we define the defect of $\Lambda$ to be equal to the defect of $\cE_0\otimes_{\cO_0}\Lambda$ (which is defined as per the previous point).
        \end{enumerate}
    \end{defi}
    \begin{rem}
        \begin{enumerate}
        \item Note that the defect of a symmetric order $\Lambda$ is well-defined (i.e. independent of the choice of $T$ and the splitting field $E_0$).

        Independently of whether $K_0$ is a splitting field for $A$, a symmetrising form $T$ defines an isomorphism $$\Lambda\longrightarrow \Hom_\cO(\Lambda, \cO):\ a \mapsto T(a,-)$$
        of $\Lambda$-$\Lambda$-bimodules. Such an isomorphism is clearly unique up to an automorphism of $\Lambda$ viewed as a $\Lambda$-$\Lambda$-bimodule, and such automorphisms are given by multiplication by an element of $Z(\Lambda)^\times$.
        So if $T'$ is another symmetrising form for $\Lambda$, then $T'(-,=) =T(z\cdot -, =)$ for some $z \in Z(\Lambda)^\times$. If $K_0$ is a splitting field for $A$, then for all $i$ and all $a,b\in A$ we have $\Tr_i(zab)=z_i\Tr_i(ab)$ for some $z_i\in \cO_0^\times$ (using the notation of Definition~\ref{defi defect}). In particular, the $u_i$ attached to the forms $T$ and $T'$ differ only by an element of $\cO_0^\times$, that is, they have the same valuation.

        The above argument shows that the defect of a symmetric order in a split semisimple algebra is independent of the choice of a symmetrising form. The second part of Definition~\ref{defi defect} defines the defect in the non-split case by passing to a splitting field.
        So assume that we have two different splitting fields $E_0$ and $E_0'$, both of finite degree over $K_0$. We need to show that the defect of $\Lambda$ is independent of whether we use $E_0$ or $E_0'$ as our splitting field in Definition~\ref{defi defect}.
        Fix an algebraic closure $\bar K_0$ of $K_0$.
         We can choose embeddings $i:\ E_0\hookrightarrow \bar K_0$ and $i':\ E_0'\hookrightarrow \bar K_0$.
        Then there is a bigger splitting field $E_0''\subset \bar K_0$ containing both $i(E_0)$ and $i(E_0')$. As the valuation $\nu_p$ on $K_0$ extends uniquely to any finite  algebraic extension, we have $\nu_p(i(x))=\nu_p(x)$ for all $x\in E_0$ (same for $i'$ and $E_0'$). Hence we may replace, without loss of generality, $E_0$ by $i(E_0)$ and $E_0'$ by $i'(E_0')$ and just assume that $E_0$ and $E_0'$ are contained (as discretely valued fields) in $E_0''$.
         The explicit symmetrising forms $T$ and $T'$ we chose over $E_0$ and $E_0'$  both extend linearly to symmetrising forms over $E_0''$. The invariants $u_i$ used in Definition~\ref{defi defect} for $T$ (respectively $T'$) are the same as for the $E_0''$-linear extension of $T$ (respectively $T'$). That is, the defect of $\Lambda$ obtained using the splitting field $E_0$ (respectively $E_0'$) is the same as the one obtained using the splitting field $E_0''$. As we have seen in the previous paragraph that the defect of an order in a split semi-simple algebra over a given field is well-defined, it follows that defect of $\Lambda$ defined using the splitting fields $E_0$ or $E_0'$ is the same.
\item Let $E_0/K_0$ be an extension of discretely valued fields, and let $\mathcal E_0$ and $\cO_0$ denote the corresponding discrete valuation rings. If $\Lambda$ is an $\cO_0$-order in a semisimple $K_0$-algebra, then the defect of the $\cO_0$-order $\Lambda$ is the same as the defect of the  $\mathcal E_0$-order $\mathcal E_0\otimes_{\cO_0}\Lambda$ (this is again easy to see, one just needs to construct a finite splitting extension of $E_0$ containing a finite splitting extension of $K_0$).
\item
 If $e \in \Lambda$ is an idempotent, and $T:\ A \times A \longrightarrow K_0$ is a symmetrising form for $\Lambda$, then $T|_{eAe\times eAe}$ is a symmetrising form for $e\Lambda e$. In particular, if $e$ does not annihilate any non-zero element of $Z(A)$, then $\Lambda$ and $e\Lambda e$ have the same defect (this can be seen by passing to a splitting field). It follows that the basic algebra of $\Lambda$ has the same defect as $\Lambda$, that is, the defect is invariant under Morita equivalence.
%
        \item Let $\Lambda = \cO_0 G$, and assume without loss of generality that $K_0 G$ is split. If $\chi_1,\ldots, \chi_n: K_0G \longrightarrow K_0$ are the (absolutely) irreducible characters of $G$, then $\chi_i=\Tr_i$ (up to permutation of the indices). It is easy to see that
        $\cO_0 G$ is self-dual with respect to the bilinear form
        $T(a, b)=|G|^{-1}\cdot \chi_{\rm reg}(a\cdot b)$, where $\chi_{\rm reg}$ denotes the character of the regular representation of $G$.
        We have $\chi_{\rm reg}= \sum_i \chi_i(1)\cdot \chi_i$, and therefore
        $$
            T(a,b) = \frac{1}{|G|}\cdot \sum_{i=1}^n \chi(1)\cdot \Tr_i(a\cdot b).
        $$
        That is $u_i = |G|^{-1}\cdot \chi_i(1)$. In particular, the defect of $\cO_0G$ is equal to $\nu_p(|G|)$.
        \item If $\Lambda=\cO_0G b$ is a block, then the above reasoning implies
        that the defect of $\Lambda$ in the sense of Definition~\ref{defi defect} is equal to $$\max_{\chi\in \Irr_{\CC}(b)} \{\nu_p (|G|) - \nu_p(\chi(1)) \}$$
        This equals the defect of $\cO_0 G b$ in the ordinary sense (that is, the $p$-valuation of the order of a defect group) since any block contains a character of height zero.
        \end{enumerate}
    \end{rem}

    \begin{prop}\label{prop bound index}
        Let $K_0/\QQ_p$ be an extension of discretely valued fields and let $\cO_0$ be the associated discrete valuation ring.
        Let $A$ be a semisimple $K_0$-algebra and let $\Gamma \subset A$ be a maximal $\cO_0$-order in $A$ (unique up to conjugation).
        If $\Lambda\subseteq \Gamma$ is a symmetric $\cO_0$-order of defect $d$, then
        $$
            \length_{\cO_0} (\Gamma / \Lambda) \leq \frac{1}{2}\cdot e \cdot d\cdot \dim_{K_0} (A)
        $$
        where $e=\length_{\cO_0}(\cO_0/p\cO_0)$.
    \end{prop}
    \begin{proof}
        If $A$ is split then this follows immediately from equation \eqref{eqn length estimate}. If $A$ is not split, $E_0$ is a finite extension of $K_0$ which splits $A$ and $\cE_0$ is the integral closure of $\cO_0$ in $E_0$, then
        $$
            \length_{\cO_0}(\Gamma/\Lambda) = f^{-1} \cdot \length_{\cE_0} (\cE_0\otimes_{\cO_0}\Gamma / \cE_0\otimes_{\cO_0}\Lambda) \leq f^{-1} \cdot \length_{\cE_0}(\tilde \Gamma / \cE_0\otimes_{\cO_0}\Lambda)
        $$
        where $\tilde \Gamma$ is a maximal $\cE_0$-order containing
        $\cE_0\otimes_{\cO_0}\Gamma$ and $f=\length_{\cE_0}(\cE_0 / \rad (\cO_0)\cE_0) \geq 1$ (the ramification index of the extension $E_0/K_0$). The right hand side can be bounded using equation~\eqref{eqn length estimate} as before, so
        $$
            \length_{\cO_0}(\Gamma/\Lambda) \leq \frac{1}{2}\cdot f^{-1} \cdot e\cdot d \cdot \dim_{E_0}(E_0\otimes_{K_0}A) \leq \frac{1}{2}\cdot e\cdot d \cdot {\dim_{K_0}(A)}.
        $$
    \end{proof}

    \begin{thm}\label{thm finite class}
        Fix $d, n \in \NN$.
        Up to isomorphism there are only finitely many symmetric $\cO$-orders $\Lambda$ satisfying all of the following:
        \begin{enumerate}
            \item $\dim_K(K\otimes_\cO \Lambda) \leq n$.
            \item The defect of $\Lambda$ is $d$.
            \item $\Lambda \cong \Lambda^{(\sigma)}$ as $\cO$-algebras.
        \end{enumerate}
    \end{thm}
    \begin{proof}
        Define $K_0=K^{\langle\sigma\rangle}$ and $\cO_0=\cO^{\langle\sigma\rangle}$.
        By Theorem~\ref{thm ex o0 form} any $\Lambda$ satisfying the conditions above has an $\cO_0$-form $\Lambda_0$. By Proposition~\ref{prop finitely many cyclotomic split} there are only finitely many $K_0$-algebras which can occur as the $K_0$-span of $\Lambda_0$. Hence it suffices to show that any semisimple $K_0$-algebra $A_0$ contains only finitely many isomorphism classes of symmetric $\cO_0$-orders of defect $d$.

        The algebra $A_0$ contains a maximal order $\Gamma_0$ which is unique up to conjugation. By Proposition~\ref{prop bound index} the $\cO_0$-length of the quotient $\Gamma_0/\Lambda_0$ for a symmetric $\cO_0$-order $\Lambda_0\subseteq \Gamma_0$  of defect $d$ is bounded by ${\frac{1}{2}\cdot e\cdot d \cdot n}$, where $e=\length_{\cO_0}(\cO_0/p\cO_0)$.  Now we just need to realise that there are only finitely many isomorphism classes of $\cO_0$-modules of length smaller than this bound (as the residue field of $\cO_0$ is finite), and for each of these (torsion) $\cO_0$-modules the set of $\cO_0$-homomorphisms from $\Gamma_0$ onto the module is a finite set. Any $\Lambda_0$ occurs as the kernel of such a homomorphism, which proves that there are only finitely many possibilities.
    \end{proof}

    \begin{thm}\label{thm finite blocks}
        Let $c,d,m \in \NN$.
        There are only finitely many Morita equivalence classes of blocks of finite groups (defined over $\cO$) such that
        \begin{enumerate}
            \item The sum of all entries of the Cartan matrix of the block is bounded by $c$.
            \item The defect of the block is equal to $d$.
            \item The $\cO$-Morita-Frobenius number of the block is bounded by $m$.
        \end{enumerate}
    \end{thm}
    \begin{proof}
        Consider the basic algebra of such a block, note that this is also symmetric with the same defect. The bound on the Cartan numbers implies a bound on the dimension of the $K$-span of the basic algebra. Moreover, a Morita equivalence of blocks corresponds to an isomorphism of basic algebras.
        Let $n$ denote the least common multiple of the integers between $1$ and $m$, and let $\sigma$ be a lift of the $n^{\nth}$ power of the Frobenius automorphism of $k$.
         Any basic algebra $\Lambda$ of a block satisfying the third condition will satisfy $\Lambda \cong {\Lambda^{(\sigma)}}$, since $n$ is a multiple of the $\cO$-Morita-Frobenius number of $\Lambda$. It follows that the collection of basic algebras of the blocks satisfying the three conditions satisfies the conditions of Theorem~\ref{thm finite class} (for the chosen $\sigma$). Hence this collection contains only finitely many isomorphism classes of orders, which implies the assertion of the theorem.
    \end{proof}

\begin{cor}
\label{kessar_analogue}
Let $\cX$ be a collection of $\cO$-blocks of finite groups and let $P$ be a finite $p$-group. The following are equivalent:
\begin{enumerate}
\item Conjecture \ref{Donovan:conj} holds for $P$ for blocks in $\cX$, that is, there are only finitely many Morita equivalence classes amongst the blocks in $\cX$ with defect group isomorphic to $P$.

\item Conjectures \ref{Weak_Donovan:conj} and \ref{O-Morita-Frobenius:conj} both hold for $P$ for blocks in $\cX$.

\end{enumerate}
\end{cor}

\begin{proof}
By Proposition \ref{morita_O_proposition}, which relates the strong $\cO$-Frobenius number to the $\cO$-Morita-Frobenius number, this follows from Theorem \ref{thm finite blocks}.
\end{proof}


\section{Reductions for Donovan's conjecture}
\label{reductions}

The general strategy for the reduction for Donovan's conjecture is the same as that in~\cite{el18b}, where the reduction proceeds in two steps. First it is shown that it suffices to consider \emph{reduced pairs}, and then it is shown that in order to demonstrate the conjecture for reduced pairs, we need only consider quasisimple groups. In~\cite{el18b} the first part of the reduction, to reduced pairs, could only be achieved over $k$ since it relied on the results of~\cite{ku95}. However the analogue of the results of~\cite{ku95} has since been shown by the second author. The following comes from~\cite[Corollary 4.18]{eis18}.

\begin{thm}[\cite{eis18}]
\label{Kulshammer_O:theorem}
Let $P$ be a finite $p$-group. Given a finite group $G$ and a block $B$ of $\cO G$ with defect group $D \cong P$ covering a block $C$ of $\cO H$ where $H=\langle D^h:h \in H \rangle$, there are only finitely many possibilities for the Morita equivalence class of $B$ given that of $C$.
\end{thm}

Recall that for a finite group $G$, a normal subgroup $N$ and a $G$-stable block $b$ of $\cO N$, we define $G[kb]$ to be the normal subgroup of $G$ consisting of elements which act as inner automorphisms on $kb$. By~\cite[Propisition 3.1]{el18b}, if $b$ is covered by a block of $\cO G$ with abelian defect group $D$, then $D \leq G[kb]$.

We recall the definition and some properties of the generalized Fitting subgroup $F^*(H)$ of a finite group $H$. Details may be found in~\cite{asc00}. A \emph{component} of $H$ is a subnormal quasisimple subgroup of $H$. Distinct components of $H$ commute and so if $L_1,L_2$ are two components of $H$ then $L_1\cap L_2\subseteq Z(L_1)$. We define the \emph{layer} $E(H)$ of $H$ to be the normal subgroup of $H$ generated by the components. It is a central product of the components. The \emph{Fitting subgroup} $F(H)$ is the largest nilpotent normal subgroup of $H$, and this is the direct product of $O_r(H)$ for all primes $r$ dividing $|H|$. The \emph{generalized Fitting subgroup} $F^*(H)$ is $E(H)F(H)$. An important property $F^*(H)$ is that $C_H(F^*(H)) \leq F^*(H)$, so $G/F^*(H)$ may be viewed as a subgroup of $\Out(F^*(H))$.

Our definition of reduced pairs is as in~\cite{el18b}:

\begin{defi}
Let $G$ be a finite group and $B$ a block of $\cO G$ with defect group $D$. We call $(G,B)$ a \emph{reduced pair} if it satisfies the following:
\begin{enumerate}
\item[(R1)] $F(G)=Z(G)=O_p(G)O_{p'}(G)$;
\item[(R2)] $O_{p'}(G) \leq [G,G]$;
\item[(R3)] Every component of $G$ is normal in $G$;
\item[(R4)] If $L \leq G$ is a component, then $L \cap D \nsubseteq Z(L)$;
\item[(R5)] $G= \langle D^g:g \in G \rangle$;
\item[(R6)] If $H$ is any characteristic subgroup of $G$, then $B$ covers a unique (i.e., $G$-stable) block $b$ of $\cO H$ and further $G[kb]=G$.
\end{enumerate}
\end{defi}

We now give the first part of the reduction, which is analogous to~\cite[Proposition 3.4]{el18b} and based on~\cite{du04}. We give a proof here for completeness.

\begin{prop}
\label{duvel}
Let $P$ be an abelian $p$-group for a prime $p$. In order to verify Donovan's conjecture for $P$, it suffices to verify that there are only a finite number of Morita equivalence classes of blocks $B$ of $\cO G$ with defect group $D \cong P$ occurring in reduced pairs $(G,B)$.
\end{prop}

\begin{proof}
Fix a finite abelian $p$-group $P$.

Consider pairs $([G:O_{p'}(Z(G))],|G|)$ with the lexicographic ordering, where $G$ is a finite group. We first use two processes, labelled (a) and (b), which we apply alternately to $\cO$-blocks of finite groups with defect groups isomorphic to $P$. Both processes strictly reduce $([G:O_{p'}(Z(G))],|G|)$ when applied non-trivially, hence after repeated application must terminate.

Let $G$ be a finite group and $B$ be a block of $\cO G$ with defect group $D \cong P$.

\begin{enumerate}[(a)]
\item Suppose $N \lhd G$ and $b$ is a block of $\cO N$ covered by $B$. Write $I=I_G(b)$ for the stabilizer of $b$ in $G$, and $B_I$ for the Fong-Reynolds correspondent. Now $B_I$ is Morita equivalent to $B$ and they have isomorphic defect groups. Since $O_{p'}(Z(G)) \leq O_{p'}(Z(I))$, if $I \neq G$, then $[I:O_{p'}(Z(I))]<[G:O_{p'}(Z(G))]$. Process (a) involves replacing $B$ by $B_I$ and repeating the process until $B$ is necessarily quasiprimitive.
\item Assume that process (a) has been performed, which means that $B$ is quasiprimitive. Let $N \lhd G$ such that $N \not\leq (Z(G) \cap [G,G])O_p(G)$, and suppose that $B$ covers a nilpotent block $b$ of $N$. Let $b'$ be a block of $Z(G)N$ covered by $B$ and covering $b$. Since $B$ is quasiprimitive both $b$ and $b'$ are $G$-stable. Further $b'$ must also be nilpotent. Using the results of~\cite{kp90}, as outlined in~\cite[Proposition 2.2]{ekks14}, $B$ is Morita equivalent to a block $\tilde{B}$ of a central extension $\tilde{L}$ of a finite group $L$ by a $p'$-group (which further is contained in the derived subgroup $[\tilde{L},\tilde{L}]$) such that there is an $M \lhd \tilde{L}$ with $M \cong D \cap (Z(G)N)$, $G/Z(G)N \cong \tilde{L}/Z(\tilde{L})M$, and $\tilde{B}$ has defect group isomorphic to $D$. Note that $[\tilde{L}:O_{p'}(Z(\tilde{L}))] \leq |L| = [G:Z(G)N]|D \cap (Z(G)N)| < [G:O_{p'}(Z(G))]$ and that $M \leq (Z(\tilde{L}) \cap [\tilde{L},\tilde{L}])O_p(\tilde{L})$. Process (b) consists of replacing $G$ by $\tilde{L}$ and $B$ by $\tilde{B}$.
\end{enumerate}


By repeated application of (a) and (b) to all blocks of all normal subgroups we have that $B$ is Morita equivalent to a quasiprimitive block $C$ (i.e., every covered block of every normal subgroup is stable) of a finite group $H$ in which $N \leq Z(H)O_p(H)$ and $O_{p'}(Z(H)) \leq [H,H]$ whenever $C$ covers a nilpotent block of a normal subgroup $N$ of $H$. Hence in order to prove Donovan's conjecture it suffices to consider such blocks.

Let $G$ be a finite group and $B$ a quasiprimitive block of $\cO G$ with defect group $D \cong P$ such that $N \leq Z(G)O_p(G)$ whenever $B$ covers a nilpotent block of a normal subgroup $N$ of $G$.

Let $H=\langle D^g: g \in G \rangle \lhd G$. Let $C$ be the unique block of $\cO H$ covered by $B$. If $N$ is a characteristic subgroup of $H$, then $N \lhd G$ and if $b$ is a block of $N$ covered by $C$, then $b$ is covered by $B$. Hence if $b$ is a block of a characteristic subgroup of $H$ covered by $C$, then $b$ is $G$-stable. Further, $D \leq G[kb] \lhd G$ by~\cite[Proposition 3.1]{el18b}, so $H \leq G[kb]$.  If further $b$ is nilpotent, then $N \leq Z(G)O_p(G)$ (and $N \leq Z(H)O_p(H)$).

We claim that $(H,C)$ is reduced. Note that we have already shown that it satisfies (R2), (R5) and (R6). Since $D$ is abelian and contains any normal $p$-subgroup of $H$ (and so in particular $O_p(H)$) we have $D \leq C_H(O_p(H)) \lhd H$, so $C_H(O_p(H))=H$, i.e., $O_p(H) \leq Z(H)$. Since also $O_{p'}(H) \leq Z(H)$ by application of (b) to $O_{p'}(H)$ we have that (R1) holds.

Write $L_1,\ldots,L_t$ for the components of $H$, so $E(H)=L_1\cdots L_t \lhd H$. Note that $H$ permutes the $L_i$. If there are no components, then $F^*(H)=Z(H)O_p(H)$ by (R1), so $D \leq C_H(F^*(H)) \leq F^*(H)=Z(H)O_p(H)$ (since $O_p(H) \leq D$ and $D$ is abelian) and $D \lhd H$, so that $H=D$, and (R3), (R4) hold. Suppose that there is at least one component.

We claim that we cannot have $D \cap L_j \leq Z(L_j)$ for any $j$. Write $L=E(H)$ and $M:=\overline{L}_1 \times \cdots \times \overline{L}_t$, where $\overline{L}_i:=L_iO_p(H)/O_p(H)$. Write $C_L$ for the unique block of $\cO L$ covered by $C$ and $\overline{C}_L$ for the unique block of $\overline{L}:=LO_p(H)/O_p(H)$ corresponding to $C_L$. There is a $p'$-group $W \leq Z(M)$ and a block $C_M$ of $\cO M$ with $W$ in its kernel such that $\overline{L} \cong M/W$ and $C_M$ is isomorphic to $\overline{C}_L$. Then $D \cap L$ is a defect group for $C_L$, $(D \cap L)O_p(H)/O_p(H)$ is a defect group for $\overline{C}_L$ and $C_M$ has defect groups isomorphic to $(D \cap L)O_p(H)/O_p(H)$. Write $c_i$ for the unique block of $L_i$ covered by $C_L$ and $\overline{c}_i$ for the unique block of $\overline{L}_i$ corresponding to $c_i$. Then $\overline{c}_i$ has defect group $D_i= ((D \cap L)O_p(H)/O_p(H)) \cap \overline{L}_i$. We have that $C_M=\overline{c}_1 \otimes \cdots \otimes \overline{c}_t$ and $C_M$ has defect group $D_1 \times \cdots \times D_t$. Let $J \subseteq \{1,\ldots,t\}$ correspond to the orbit of $L_j$ under the permutation action of $H$ on the components. Suppose that $D \cap L_j \leq Z(L_j)$ for some $j$, so $c_j$ is nilpotent. Define $L_J \lhd H$ to be the product of the $L_i$ for $i \in J$, and write $c_J$ for the unique block of $L_J$ covered by $C_L$. Now the unique block $\overline{c}_J$ of $L_JO_p(G)/O_p(G)$ corresponding to $c_J$ is isomorphic to a block of $\bigtimes_{i \in J} \overline{L}_i$ with a central $p'$-group in the kernel. Products of nilpotent blocks are nilpotent, so $\overline{c}_J$ is nilpotent. Since $O_p(G) \leq Z(G)$, by~\cite{wa94} $c_J$ is also a nilpotent block, of a nonsolvable normal subgroup covered by $C$, a contradiction. Hence for all $j$ we have $D \cap L_j \nsubseteq Z(L_j)$, so (R4) holds for $(H,C)$.

Conjugation induces a permutation action on the components, hence a homomorphism $\varphi:H \rightarrow S_t$. Let $g\in D$ and say $L_i^g=L_j$ for some $i\neq j$. Since $D \cap L_i \nsubseteq Z(L_i)$ and $L_i\cap L_j\subseteq Z(L_i)$ we have a contradiction and so $D \leq \ker(\varphi)$. Now (R5) implies that $\ker (\varphi)=H$, i.e., (R3) holds for $(H,C)$, and $(H,C)$ is reduced.

Finally, by Theorem \ref{Kulshammer_O:theorem} for a fixed Morita equivalence class for $C$, there are only finitely many possibilities for the Morita equivalence class of $B$, and we are done.
\end{proof}

In the second part of the reduction, from reduced pairs to blocks of quasisimple groups, we first show that in order to bound the strong $\cO$-Frobenius number for reduced pairs it suffices to bound it for quasisimple groups. This is already given in~\cite{el18b}:

\begin{lem}[Lemma 3.5 of~\cite{el18b}]
\label{SOF_for_reduced_pairs}
If there is a function $s:\mathbb{N}\to\mathbb{N}$ such that for all $\cO$-blocks $B$ of quasisimple groups with abelian defect groups of order $p^d$, $\SOF(B) \leq s(d)$, then there is a function $r:\mathbb{N}\to\mathbb{N}$ such that for all reduced pairs $(G,B)$ of a finite group $G$ and a block  $B$ of $\cO G$ with abelian defect groups of order $p^d$ we have $\SOF(B) \leq r(d)$.
\end{lem}

The remainder of the proof of Theorem \ref{reduce:theorem} now consists of observing that bounding the strong $\cO$-Frobenius numbers for reduced pairs implies a bound on the number of Morita equivalence classes amongst reduced pairs. In~\cite{el18b}, this part of the reduction could only be achieved over $k$ since it relied on the results of~\cite{ke04}. The results of the previous section remedy this.

{\noindent \emph{Proof of Theorem \ref{reduce:theorem}.}}

Suppose that there is a function $s:\mathbb{N}\to\mathbb{N}$ such that $\SOF(b) \leq s(d)$ for all $\cO$-blocks $b$ of quasisimple groups with abelian defect group of order $p^d$. By Lemma \ref{SOF_for_reduced_pairs} $\SOF(B)$ is bounded in terms of $D$ for all reduced pairs $(G,B)$ where $B$ is a block of $\cO G$ with abelian defect group $D$. We have assumed that the Cartan invariants of the blocks of quasisimple groups with abelian defect groups are bounded in terms of the defect. Then by~\cite[Theorem 3.2]{du04} the Cartan invariants of any block with abelian defect groups are bounded in terms of the defect, and so in particular this holds for blocks $B$ for finite groups $G$ such that $(G,B)$ is reduced. Hence by Theorem \ref{kessar_analogue} there are only finitely many Morita equivalence classes amongst blocks in reduced pairs with defect group isomorphic to $P$  and by Proposition \ref{duvel} we are done.{\hfill $\Box$   \hskip - \parfillskip\bigskip}

\begin{cor}
\label{Farrell-Kessar_consequence}
Let $P$ be a finite abelian $p$-group. Suppose that there is a function $c:\NN \rightarrow \NN$ such that if $G$ is a quasisimple finite group and $B$ is a block of $kG$ with abelian defect groups $D$ of order $p^d \leq |P|$, then the entries of the Cartan matrix of $B$ are at most $c(d)$. Then Donovan's conjecture holds for $\cO$-blocks with defect groups isomorphic to $P$.
\end{cor}

\begin{proof}
This follows from Theorem \ref{reduce:theorem} and~\cite{fk18}, in which it is proved that $\SOF(B) \leq 4|D|^2!$ for all blocks $B$ of quasisimple finite groups with defect groups $D$. Note that the setting of~\cite{fk18} is that $\cO$ is the ring of Witt vectors for $k$. However for $\cO'$ a complete discrete valuation ring containing $\cO$ with $\cO'/J(\cO') \cong k$, we have ${\rm sf_{\mathcal{O}'}}(B \otimes_\cO \cO') \leq \SOF(B)$.
\end{proof}

\begin{cor}
Donovan's conjecture holds for $\cO$-blocks whose defect groups are abelian $2$-groups.
\end{cor}

\begin{proof}
This follows immediately from Corollary \ref{Farrell-Kessar_consequence} and~\cite[9.2]{ekks14}, in which it is proved that the Cartan invariants of $2$-blocks with abelian defect groups are bounded in terms of the defect.
\end{proof}

\begin{rem}
\label{O-remark}
Note that if $\cO\subseteq \cO'$ then $\operatorname{sf}_{\cO'}(B\otimes \cO')\leq \SOF(B)$ for some $\cO$-block $B$. Therefore all the results of this section hold for $\cO$ equal to the ring of Witt vectors of $k$ adjoining a primitive $|P|^{\nth}$ root of unity, where we are considering blocks with defect group isomorphic to $P$. This is a very common and natural choice of ring to work over as it guarantees that $e_\chi\in KB$ for all $\chi\in\Irr(B)$.
\end{rem}


\section{Bounding Cartan invariants}
\label{Cartan}

We are left with the difficult problem of finding a bound on the Cartan invariants of blocks of quasisimple groups in terms of the defect group, so we gather together some (known) comments on the problem. In general, there are few $p$-groups for which a bound on the Cartan invariants is known but Donovan's conjecture is not known to hold. The generalised quaternion $2$-groups are an exception, where Donovan's conjecture is still not known in the case where the block has two simple modules, but the Cartan matrix is known (see~\cite{er90}). Following~\cite{du04}, for $G$ a finite group and $B$ a block of $kG$ with defect group $D$, write $\LL(B)$ for the Loewy length of $B$, the smallest $n$ such that $\rad^n(B)=0$. Write $$e(B)=\max\{\dim_k({\rm Ext}^1_{kG}(V,W)):V,W \ {\rm simple} \ kG{\rm -modules} \}.$$ The largest Cartan invariant of $B$ is at most
$$\sum_{i=0}^{\LL(B)} e(B)^il(B)^i$$ and by the Brauer-Feit theorem $l(B) \leq |D|^2$, so bounding the Cartan invariants in terms of the defect group is equivalent to bounding $\LL(B)$ and $e(B)$ in terms of the defect group.

By~\cite[Theorem 3.4]{du04} if $Z \leq O_p(G)$ and $\overline{B}$ the unique block of $kG/Z$ corresponding to $B$, we have $\LL(B) \leq \LL(\overline{B}) \LL(kZ)$ and $e(B) \leq e(\overline{B})+e(kZ)$, so in order to prove Donovan's conjecture it now suffices to bound the Cartan invariants for blocks of quasisimple groups with no non-trivial central $p$-subgroup.

Bounds for the Loewy length in terms of the defect group have been considered in~\cite{el17} for abelian defect groups, although bounds are only demonstrated for $p=2$ and so do not contribute anything to our knowledge of Donovan's conjecture.

The question of bounding $\dim_k({\rm Ext}^1_{kG}(V,W))$ for simple $B$-modules $V$ and $W$ is related to a conjecture of Guralnick in~\cite{gu86} where it is predicted that there should be an \emph{absolute} bound when $V$ is the trivial module and $W$ is faithful. In~\cite{ps11} it is shown that for finite groups of Lie type in defining characteristic $\dim_k({\rm Ext}^1_{kG}(V,W))$ is bounded in terms of the size of the root system, with no restrictions on $V$ and $W$. Therefore, since all blocks of non-trivial defect are of full defect for finite groups of Lie type in defining characteristic, there is a bound in terms of the size of the defect group in this case.

\begin{ack*} We thank both Gunter Malle and the anonymous referee for their careful reading of the manuscript and for their helpful comments and suggestions.
\end{ack*}

\end{document}